\documentclass[10.5pt,a4paper]{article}
\usepackage{mathrsfs}
\usepackage{epsfig, graphicx}
\usepackage{latexsym,amsfonts,amsbsy,amssymb}
\usepackage{amsmath,amsthm}
\usepackage{color}
\usepackage{hyperref,color}
\makeatletter 

\@addtoreset{equation}{section}
\makeatother 
\allowdisplaybreaks 

\newtheorem{theorem}{Theorem}[section]
\newtheorem{lemma}{Lemma}[section]

\newtheorem{remark}{Remark}[section]
\newtheorem{algorithm}{Algorithm}[section]

\newtheorem{example}{Example}[section]

\begin{document}
\title{An Efficient Multigrid Method for Ground
State Solution of Bose-Einstein Condensates\footnote{This work was supported in part by National
Natural Science Foundations of China (NSFC 91730302, 11771434, 91330202, 11371026, 11001259,
11031006, 2011CB309703), Science Challenge Project (No. TZ2016002),
the National Center for Mathematics and Interdisciplinary Science, CAS.}}
\author{
Hehu Xie\footnote{LSEC, ICMSEC,
    Academy of Mathematics and Systems Science, Chinese Academy of
    Sciences, Beijing 100190, China, and School of Mathematical
    Sciences, University of Chinese Academy of Sciences, Beijing,
    100049, China (hhxie@lsec.cc.ac.cn)},\ \ \
Fei Xu\footnote{Beijing Institute for Scientific and Engineering
    Computing, Beijing University of Technology, Beijing 100124, China
    (xufei@lsec.cc.ac.cn)} \ \  {\rm and} \ \
Ning Zhang\footnote{LSEC, ICMSEC,
    Academy of Mathematics and Systems Science, Chinese Academy of
    Sciences, Beijing 100190, China, and School of Mathematical
    Sciences, University of Chinese Academy of Sciences, Beijing,
    100049, China (zhangning114@lsec.cc.ac.cn)}
}
\date{}
\maketitle
\begin{abstract}
An efficient multigrid method is proposed to compute the ground state solution of
Bose-Einstein condensations by the finite element method based on the combination of
the multigrid method for nonlinear eigenvalue problem and an efficient implementation for the nonlinear iteration.
The proposed numerical method not only has the optimal convergence rate, but also
has the asymptotically optimal computational work which is independent from the nonlinearity of the problem.
The independence from the nonlinearity means that the asymptotic estimate of the computational work
can reach almost the same as that of solving the corresponding linear boundary value problem by the multigrid method.
Some numerical experiments are provided to validate the efficiency of the proposed method.
\end{abstract}
\section{Introduction}
It is well known that Bose-Einstein condensation (BEC), which is a gas of bosons that are in the same quantum state,
is an important and active field \cite{AnderEnsherMattewWieman,AnglinKetterle,BaoCai,CornellWieman,Ketterle} in physics.
The properties of the condensate at zero or very low temperature \cite{DalGioPitaString,LiebSeiYang}
can be described by the well-known Gross-Pitaevskii equation (GPE) \cite{Gross}
which is a time-independent nonlinear Schr\"odinger equation \cite{LaudauLifschitz}.

Since this paper considers the numerical method for the nonlinear eigenvalue problem,
we are concerned with the following non-dimensionalized GPE problem: Find $\lambda\in\mathbb R$ and a function $u$ such that
\begin{equation}\label{GPEsymply2}
\left\{
\begin{array}{rcl}
-\Delta u + Wu + \zeta |u|^2u &=& \lambda u,\ \ \  {\rm in}\  \Omega,\\
u &=& 0,\ \ \ \ \  {\rm on}\ \partial \Omega,\\
\int_\Omega|u|^2d\Omega&=& 1,
\end{array}
\right.
\end{equation}\\where $\Omega \subset \mathbb{R}^d$ $(d = 1,2,3)$ denotes the computing
domain which has the cone property \cite{Adams}, $\zeta$ is some positive constant and
$W(x) = \gamma_1 x^2_1 +\ldots + \gamma_d x^2_d \geq 0$
with  $\gamma_1, \ldots, \gamma_d > 0$ \cite{BaoTang,ZhouBEC}.

The convergence of the finite element method for GPEs is first proved in \cite{ZhouBEC} and  \cite{CancesChakirMaday} gives
prior error estimates which will be used in the analysis of our method.
There also exist two-grid finite element methods for GPE in \cite{ChienHuangJengLi,ChienJeng,HenningMalqvistPeterseim}.
Recently, a type of multigrid method for eigenvalue problems has been proposed in
\cite{LinXie,Xie_Steklov,Xie_Nonconforming,Xie_JCP,XieXie_CICP}. Especially, \cite{XieXie_CICP}
gives a multigrid method for GPE (\ref{GPEsymply2}) and the corresponding error estimates.
This type of multigrid method is designed based on the multilevel correction method in \cite{LinXie},
and a sequence of nested finite element spaces with different levels
of accuracy which can be built in the same way as the multilevel method for boundary
value problems \cite{Xu}. The corresponding error estimates have already been obtained in \cite{XieXie_CICP}.
Furthermore, the estimate of computational work has also been given in \cite{XieXie_CICP}.
The computational work of the multigrid in \cite{XieXie_CICP} is linear scale but depends on the nonlinearity (i.e. the value of $\zeta$)
in some sense.  The aim of this paper is to improve the efficiency further with a special implementing method 
for the multigrid method for the GPE. With the proposed implementing technique, the multigrid method can really 
arrive the asymptotically optimal computational complexity which is almost independent of` the nonlinearity of the GPE.

An outline of the paper goes as follows. In Section 2, we introduce finite element
method for the ground state solution of BEC, i.e. non-dimensionalized GPE (\ref{GPEsymply2}).
A type of one correction step is given in Sections 3.
In Section 4, we propose an efficient implementing technique for the nonlinear eigenvalue problem
included in the one correction step. A type of multigrid algorithm for solving the non-dimensionalized GPE
by the finite element method will be stated in Section 5. Three numerical examples are provided in Section 6
to validate the efficiency of the proposed numerical method in this paper.
Some concluding remarks are given in the last section.

\section{Finite element method for GPE problem}
This section is devoted to introducing some notation and finite element method
for the GPE (\ref{GPEsymply2}). The letter $C$ (with or without subscripts) denotes
a generic positive constant which may be different at its different occurrences.
For convenience, the symbols $\lesssim$, $\gtrsim$ and $\approx$
will be used in this paper. That $x_1\lesssim y_1, x_2\gtrsim y_2$
and $x_3\approx y_3$, mean that $x_1\leq C_1y_1$, $x_2 \geq c_2y_2$
and $c_3x_3\leq y_3\leq C_3x_3$ for some constants $C_1, c_2, c_3$
and $C_3$ that are independent of mesh sizes (see, e.g., \cite{Xu}).
The standard notation for the Sobolev spaces $W^{s,p}(\Omega)$ and their
associated norms $\|\cdot\|_{s,p,\Omega}$ and seminorms $|\cdot|_{s,p,\Omega}$
(see, e.g., \cite{Adams}) will be used. For $p=2$, we denote
$H^s(\Omega)=W^{s,2}(\Omega)$ and $H_0^1(\Omega)=\{v\in H^1(\Omega):\ v|_{\partial\Omega}=0\}$,
where $v|_{\partial\Omega}=0$ is in the sense of trace and 
$\|\cdot\|_{s,\Omega}=\|\cdot\|_{s,2,\Omega}$. In this paper, we set $V=H_0^1(\Omega)$
and use $\|\cdot\|_s$ to denote $\|\cdot\|_{s,\Omega}$ for simplicity.

For the aim of finite element
discretization, we define the corresponding weak form for (\ref{GPEsymply2}) as follows:
Find $(\lambda,u)\in \mathbb{R}\times V$ such that $b(u,u) = 1$ and
\begin{equation}\label{GPEweakform}
a(u,v) = \lambda b(u,v),\ \ \  \forall v \in V,
\end{equation}
where
\[
a(u,v) := \int_\Omega\big(\nabla u\nabla v + Wuv + \zeta|u|^2 uv\big)d\Omega,
\ \ \ \ b(u,v) := \int_\Omega uvd\Omega.
\]

Now, let us define the finite element method \cite{BrennerScott,Ciarlet}
for the problem (\ref{GPEweakform}). First we
generate a shape-regular decomposition of the computing domain
$\Omega \subset \mathbb{R}^d$ $(d = 2,3)$ into triangles or rectangles for $d = 2$
(tetrahedrons or hexahedrons for $d = 3$). The diameter of a cell $K \in \mathcal{T}_h$ is
denoted by $h_K$ and define $h$ as $h:=\max_{K\in \mathcal T_h}h_K$.
Then the corresponding linear finite element space $V_h\subset V$
can be built on the mesh $\mathcal{T}_h$. We assume that $V_h\subset V$
is a family of finite-dimensional spaces that satisfy the following assumption:
\begin{equation}\label{approximation_fem}
\lim_{h\rightarrow 0}\inf_{v_h \in V_h} \|w - v_h\|_1 = 0,\ \ \ \forall w\in V.
\end{equation}

The standard finite element method for (\ref{GPEweakform}) is to solve the following
 eigenvalue problem:
Find $(\bar{\lambda}_h,\bar{u}_h)\in \mathbb{R}\times V_h$ such that
 $b(\bar{u}_h,\bar{u}_h) = 1$ and
\begin{equation}\label{GPEfem}
a(\bar{u}_h,v_h) = \bar{\lambda}_h b(\bar{u}_h,v_h),\ \ \   \forall v_h \in V_h.
\end{equation}
Then we define
\begin{equation}\label{delta}
\delta_h(u) := \inf_{v_h\in V_h}\|u - v_h\|_1.
\end{equation}
For understanding the multigrid method in this paper, we state the error estimates of the finite element
method for GPE (\ref{GPEsymply2}).
\begin{lemma}(\cite[Theorem 1]{CancesChakirMaday},\cite{ZhouBEC})\label{lemma:Maday}
There exists $h_0 > 0$, such that for all $0 < h < h_0$, the smallest eigenpair approximation
 $(\bar{\lambda}_h,\bar{u}_h)$ of (\ref{GPEfem}) has following error estimates:
\begin{eqnarray}
\|u - \bar{u}_h\|_1 &\lesssim& \delta_h(u),\\
\|u - \bar{u}_h\|_0 &\lesssim& \eta_a(V_h)\|u - \bar{u}_h\|_1 \lesssim \eta_a(V_h)\delta_h(u),\\
|\lambda - \bar{\lambda}_h| &\lesssim& \|u - \bar{u}_h\|^2_1
+ \|u - \bar{u}_h\|_0\lesssim \eta_a(V_h)\delta_h(u),
\end{eqnarray}
where $\eta_a(V_h)$ is defined as follows:
\begin{eqnarray}\label{eta_a_h}
\eta_a(V_h)=\|u-\bar{u}_h\|_1+ \sup_{f\in L^2(\Omega),\|f\|_0=1}\inf_{v_h\in V_h}\|Tf-v_h\|_1
\end{eqnarray}
with the operator $T$ being defined as follows:
 Find $Tf\in u^{\perp}$ such that
\vskip-0.7cm
\begin{eqnarray*}
a(Tf,v)+2(\zeta |u|^2(Tf),v)- (\lambda(Tf),v)=(f,v),\ \ \ \ \forall v\in u^{\perp},
\end{eqnarray*}
where $u^{\perp}=\big\{v\in H_0^1(\Omega)| \int_{\Omega}uvd\Omega=0\big\}$.
\end{lemma}

\section{One correction step}\label{sec:1correction_fix-point}
In this section, we recall the  one correction step from \cite{XieXie_CICP} to
improve the accuracy of the given eigenpair approximation. This correction step contains
solving an auxiliary linear boundary value problem with multigrid method in the finer finite
element space and a GPE on a very low dimensional finite element space which will be discussed in the next section.

In order to define the one correction step, we introduce a very coarse mesh $\mathcal{T}_H$
and the low dimensional linear finite element space $V_H$ defined on the mesh $\mathcal{T}_H$.
Assume we have obtained an eigenpair approximation $(\lambda_{h_k},u_{h_k})\in
\mathbb{R}\times V_{h_k}$ and the coarse space $V_H$ is a subset of $V_{h_k}$.
Let $V_{h_{k+1}} \subset V$
be a finer finite element space such that $V_{h_k}\subset V_{h_{k+1}}$.
Based on this finer finite element space, we define the following one correction step.
\begin{algorithm}\label{Algm:One_Step_Correction}
One Correction Step
\begin{enumerate}
\item Define the following auxiliary boundary value problem:
Find $\widehat{u}_{h_{k+1}} \in V_{h_{k+1}}$ such that
\begin{eqnarray}\label{Aux_Source_Problem}
&&(\nabla\widehat{u}_{h_{k+1}},\nabla v_{h_{k+1}})+(W\widehat{u}_{h_{k+1}},v_{h_{k+1}})
+(\zeta|u_{h_k}|^2\widehat{u}_{h_{k+1}},v_{h_{k+1}}) \nonumber\\
&&\quad\quad\quad\quad\quad\quad \quad\quad\quad \quad\quad\quad \quad\quad\quad  =
\lambda_{h_k}b(u_{h_k},v_{h_{k+1}}),\ \ \
 \forall v_{h_{k+1}}\in V_{h_{k+1}}.
\end{eqnarray}
Solve this equation with multigrid method
\cite{Bramble,BrennerScott,Hackbush,McCormick,Xu}
to obtain an approximation
$\widetilde{u}_{h_{k+1}} \in V_{h_{k+1}}$ with the error estimate
$\|\widetilde{u}_{h_{k+1}} - \widehat{u}_{h_{k+1}}\|_1 \lesssim \varsigma_{h_{k+1}}$.
Here $\varsigma_{h_{k+1}}$ is used to denote the accuracy for the multigrid iteration.

\item Define a new finite element space
$V_{H,h_{k+1}} = V_H + {\rm span}\{\widetilde{u}_{h_{k+1}}\}$ and solve the
following nonlinear eigenvalue problem:
Find $(\lambda_{h_{k+1}},u_{h_{k+1}}) \in \mathbb{R} \times V_{H,h_{k+1}}$
such that $b(u_{h_{k+1}},u_{h_{k+1}}) = 1$ and
\begin{eqnarray}\label{simple_Eigen_Problem}
a(u_{h_{k+1}},v_{H,h_{k+1}}) = \lambda_{h_{k+1}}b(u_{h_{k+1}},v_{H,h_{k+1}}),
\ \  \forall v_{H,h_{k+1}} \in V_{H,h_{k+1}}.
\end{eqnarray}
\end{enumerate}
Summarize above two steps into
\begin{eqnarray*}
(\lambda_{h_{k+1}},u_{h_{k+1}}) = Correction(V_H,\lambda_{h_k},u_{h_k},V_{h_{k+1}},\varsigma_{h_{k+1}}).
\end{eqnarray*}
\end{algorithm}

Similarly, we also state the following error estimates from \cite{XieXie_CICP}
for the one correction step defined in Algorithm \ref{Algm:One_Step_Correction}.
\begin{theorem}\label{Thm:Error_Estimate_One_Step_Correction}(\cite[Theorem 3.1]{XieXie_CICP})
Assume $h_k<h_0$ (as in Lemma \ref{lemma:Maday}) and there exists a real number $\varepsilon_{h_k}(u)$ such that
the given eigenpair approximation $(\lambda_{h_k},u_{h_k})\in\mathbb{R}\times V_{h_k}$
has the following error estimates:
\begin{eqnarray}
\|\bar{u}_{h_k}-u_{h_k}\|_{0}+|\bar{\lambda}_{h_k}-\lambda_{h_k}|
= \varepsilon_{h_k}(u).\label{Estimate_lambda_lambda_h_k}
\end{eqnarray}
Then after one correction step, the resultant approximation
$(\lambda_{h_{k+1}},u_{h_{k+1}})\in\mathbb{R}\times V_{h_{k+1}}$ has the
following error estimates:
\begin{eqnarray}
\|\bar{u}_{h_{k+1}}-u_{h_{k+1}}\|_{1} &\lesssim& \varepsilon_{h_{k+1}}(u),
\label{Estimate_u_u_h_{k+1}}\\
\|\bar{u}_{h_{k+1}}-u_{h_{k+1}}\|_{0} &\lesssim& \eta_{a}(V_H)\|u-u_{h_{k+1}} \|_{1},
\label{Estimate_u_u_h_{k+1}_negative}\\
|\bar{\lambda}_{h_{k+1}}-\lambda_{h_{k+1}}|&\lesssim&
\eta_{a}(V_H)\varepsilon_{h_{k+1}}(u),\label{Estimate_lambda_lambda_h_{k+1}}
\end{eqnarray}
where
$\varepsilon_{h_{k+1}}(u):= \eta_{a}(V_{h_k})\delta_{h_k}(u)
+\|\bar{u}_{h_k}-u_{h_k}\|_{0}+|\bar{\lambda}_{h_k}-\lambda_{h_k}|
+\varsigma_{h_{k+1}}$.
\end{theorem}
\section{Efficient implementation}
In this section, we show an efficient implementing method for Step 2 of
Algorithm \ref{Algm:One_Step_Correction}, i.e., solving the nonlinear eigenvalue problem (\ref{simple_Eigen_Problem}).
For simplicity of notation, we use $h$ to denote $h_{k+1}$. Then $V_h$, $\widetilde u_h$ and
$V_{H,h}=V_H+{\rm span}\{\widetilde u_h\}$ denote $V_{h_{k+1}}$, $\widetilde u_{h_{k+1}}$
and $V_{H,h_{k+1}}=V_H+{\rm span}\{\widetilde u_{h_{k+1}}\}$, respectively, in this section.
Here we also define $N_H:={\rm dim}V_H$ and $N_h:={\rm dim}V_h$.
Let  $\{\phi_{k,H}\}_{1\leq k\leq N_H}$ denotes the Lagrange basis function for the coarse finite element space $V_H$.

For simplicity, the fixed point (self-consistent field) iteration method with dumping technique is adopted to solve
the nonlinear eigenvalue problem (\ref{simple_Eigen_Problem}).
In each nonlinear iteration, the main content is to assemble the matrices for problem (\ref{simple_Eigen_Problem})
which is defined on the special space $V_{H,h}$. The function in $V_{H,h}$ can be denoted by $u_{H,h}=u_H+\alpha \widetilde u_h$.
Solving problem (\ref{simple_Eigen_Problem}) is to obtain the function $u_H\in V_h$ and the value $\alpha\in \mathbb R$.
Let $u_H=\sum_{k=1}^{N_H}u_k\phi_{k,H}$ and define the vector $\mathbf u_H$ as $\mathbf u_H=[u_1,\cdots, u_{N_H}]^T$.

Based on the structure of the space $V_{H,h}$, the matrix version of
the eigenvalue problem (\ref{simple_Eigen_Problem}) can be written as follows
\begin{equation}\label{Eigenvalue_H_h}
\left(
\begin{array}{cc}
A_H & b_{Hh}\\
b_{Hh}^T&\xi
\end{array}
\right)
\left(
\begin{array}{c}
\mathbf u_H\\
\alpha
\end{array}
\right)
=\lambda_h
\left(
\begin{array}{cc}
M_H& c_{Hh}\\
c_{Hh}^T & \gamma
\end{array}
\right)
\left(
\begin{array}{c}
\mathbf u_H\\
\alpha
\end{array}
\right),
\end{equation}
where $\mathbf u_H\in \mathbb R^{N_H}$ and $\alpha\in\mathbb R$.

It is obvious that the matrix $M_H$, the vector $c_{Hh}$ and the scalar $\gamma$ will not change
during the nonlinear iteration process as long as we have obtained the function $\widetilde u_h$.
But the matrix $A_H$, the vector $b_{Hh}$ and the scalar $\xi$ will change during the nonlinear
iteration process. It is required to consider the efficient implementation to update the
the matrix $A_H$, the vector $b_{Hh}$ and the scalar $\xi$ since there is a function $\widetilde u_h$
which is defined on the fine mesh $\mathcal T_h$. The aim of this section is to propose an efficient method to update
the matrix $A_H$, the vector $b_{Hh}$ and the scalar $\xi$ without computation on the fine mesh $\mathcal T_h$
during the nonlinear iteration process. Assume we have a given initial value $(u_H,\alpha)\in V_H\times \mathbb R$.
Now, in order to carry out the nonlinear iteration for eigenvalue problem (\ref{Eigenvalue_H_h}), we come to consider
the computation for the matrix $A_H$, vector $b_{Hh}$ and the scalar $\xi$.

From the definitions of the space $V_{H,h}$ and the eigenvalue problem (\ref{simple_Eigen_Problem}), the matrix
$A_H$ has the following expansion
\begin{eqnarray}
(A_H)_{i,j}&=&\int_{\Omega}\nabla\phi_{i,H}\nabla\phi_{j,H}d\Omega
+\int_{\Omega}w\phi_{i,H}\phi_{j,H}d\Omega
+\int_{\Omega}\zeta(u_H+\alpha\widetilde u_h)^2\phi_{i,H}\phi_{j,H}d\Omega\nonumber\\
&:=&(A_{H,1})_{i,j}+(A_{H,2})_{i,j},
\end{eqnarray}
where
\begin{eqnarray}\label{A_H_1}
(A_{H,1})_{i,j} &=&  \int_{\Omega}\nabla\phi_{i,H}\nabla\phi_{j,H}d\Omega
+\int_{\Omega}w\phi_{i,H}\phi_{j,H}d\Omega
\end{eqnarray}
and
\begin{eqnarray}
(A_{H,2})_{i,j} &=&\int_{\Omega}\zeta(u_H+\alpha\widetilde u_h)^2\phi_{i,H}\phi_{j,H}d\Omega \nonumber\\
&=&\int_{\Omega}\zeta\big((u_H)^2+2\alpha u_H\widetilde u_h
+\alpha^2(\widetilde u_h)^2\big)\phi_{i,H}\phi_{j,H}d\Omega\nonumber\\
&=&\int_{\Omega}\zeta(u_H)^2\phi_{i,H}\phi_{j,H}d\Omega
+ 2\alpha\int_{\Omega}\zeta\widetilde u_hu_H\phi_{i,H}\phi_{j,H}d\Omega \nonumber\\
&&\ \ \ + \alpha^2\int_\Omega\zeta (\widetilde u_h)^2\phi_{i,H}\phi_{j,H}d\Omega\nonumber\\
&:=& (A_{H,2,1})_{i,j} + 2\alpha(A_{H,2,2})_{i,j} + \alpha^2(A_{H,2,3})_{i,j}.
\end{eqnarray}
It is obvious that the computational work for the matrix
\begin{eqnarray}\label{A_H_2_1}
(A_{H,2,1})_{i,j} = \int_{\Omega}\zeta(u_H)^2\phi_{i,H}\phi_{j,H}d\Omega
\end{eqnarray}
is $\mathcal O(N_H)$.
The matrices $A_{H,1}$, and $A_{H,2,3}$ which is defined by
\begin{eqnarray}\label{A_H_2_3}
(A_{H,2,3})_{i,j} = \int_\Omega \zeta(\widetilde u_h)^2\phi_{i,H}\phi_{j,H}d\Omega
\end{eqnarray}
will not change during the nonlinear iteration process.

The matrix $A_{H,2,2}$ has the following expansion
\begin{eqnarray}\label{A_H_2_2_i_j}
(A_{H,2,2})_{i,j} = \sum_{k=1}^{N_H}u_k\int_{\Omega}\zeta\widetilde u_h\phi_{k,H}\phi_{i,H}\phi_{j,H}d\Omega.
\end{eqnarray}
The expansion (\ref{A_H_2_2_i_j}) gives a hint to define a tensor $T_H$ as follows
\begin{eqnarray}\label{T_H}
(T_H)_{i,j,k} = \int_{\Omega}\zeta\widetilde u_h\phi_{k,H}\phi_{i,H}\phi_{j,H}d\Omega.
\end{eqnarray}
Then the matrix $A_{H,2,2}$ has the following computational scheme
\begin{eqnarray}\label{Tensor_Multiply}
A_{H,2,2} = T_H\cdot\mathbf u_H,
\end{eqnarray}
where $T_H\cdot\mathbf u_H$ denotes the multiplication of the tensor $T_H$
and the vector $\mathbf u_H$ corresponding to the last index $k$.
From (\ref{T_H}), it is easy to know that the dimension of the tensor $T_H$ is $\mathbb R^{N_H\times N_H\times N_H}$
and the number of nonzero elements is $\mathcal O(N_H)$.
Thus $T_H$ is a sparse tensor and the computational work for the operation (\ref{Tensor_Multiply})
is $\mathcal O(N_H)$.

Now, let us consider the computation for the vector $b_{Hh}$. From the definition of the space $V_{H,h}$ and the
problem (\ref{simple_Eigen_Problem}), the vector $b_{Hh}$ has the following expansion
\begin{eqnarray}
(b_{Hh})_i &=& \int_{\Omega}\nabla \widetilde u_h\nabla\phi_{i,H}d\Omega + \int_{\Omega}w\widetilde u_h\phi_{i,H}d\Omega
+\int_\Omega \zeta(u_H+\alpha \widetilde u_h)^2\widetilde u_h\phi_{i,H}d\Omega\nonumber\\
&:=& (b_{Hh,1})_i+ (b_{Hh,2})_i,
\end{eqnarray}
where
\begin{eqnarray}\label{b_H_h_1}
(b_{Hh,1})_i =  \int_{\Omega}\nabla \widetilde u_h\nabla\phi_{i,H}d\Omega
+ \int_{\Omega}w\widetilde u_h\phi_{i,H}d\Omega,
\end{eqnarray}
and
\begin{eqnarray}\label{b_H_h_2}
(b_{Hh,2})_i &=& \int_\Omega\zeta(u_H+\alpha \widetilde u_h)^2
\widetilde u_h\phi_{i,H}d\Omega\nonumber\\
&=&\int_\Omega\zeta \big((u_H)^2+2\alpha \widetilde u_h u_H
+\alpha^2(\widetilde u_h)^2\big)\widetilde u_h\phi_{i,H}d\Omega\nonumber\\
&=& \int_\Omega \zeta(u_H)^2\widetilde u_h\phi_{i,H}d\Omega
+ 2\alpha \int_\Omega \zeta(\widetilde u_h)^2u_H\phi_{i,H}d\Omega
+\alpha^2\int_{\Omega}\zeta(\widetilde u_h)^3\phi_{i,H}d\Omega\nonumber\\
&:=& (b_{Hh,2,1})_i + 2\alpha(b_{Hh,2,2})_i + \alpha^2(b_{Hh,2,3})_i.
\end{eqnarray}
It is obvious that the vector $b_{Hh,1}$ will not change during the nonlinear iteration process.
Thus, we only need to consider the computation for the vector $b_{Hh,2}$.

First, the computation for the vector $b_{Hh,2,1}$ can be implemented as follows
\begin{eqnarray}\label{b_H_h_2_1}
(b_{Hh,2,1})_i = \int_{\Omega}\big(\sum_{j=1}^{N_H} u_j\phi_{j,H}\big)^2\widetilde u_h\phi_{i,H}d\Omega
= \sum_{j=1}^{N_H}\sum_{k=1}^{N_H}u_ju_k\int_{\Omega}\widetilde u_h \phi_{j,H}\phi_{k,H}\phi_{i,H}d\Omega.
\end{eqnarray}
Based on the tensor $T_H$, the vector $b_{Hh,2,1}$ can be calculated by the tensor multiplication
\begin{eqnarray}\label{Tensor_Multiply_2}
b_{Hh,2,1} = (T_H\cdot\mathbf u_H)\cdot\mathbf u_H =A_{H,2,2}\mathbf u_H,
\end{eqnarray}
where $(T_H\cdot\mathbf u_H)\cdot\mathbf u_H$ denotes the multiplication of the tensor $T_H$
with the vector $\mathbf u_H$ corresponding to the last two indices $k$ and $j$.
Similarly, the computational work for the operation (\ref{Tensor_Multiply_2}) is also $\mathcal O(N_H)$.

Then the computation for $b_{Hh,2,2}$ can be done as follows
\begin{eqnarray}\label{b_H_h_2_2}
(b_{Hh,2,2})_i = \sum_{j=1}^{N_H}u_j\int_{\Omega}\zeta(\widetilde u_h)^2\phi_{j,H}\phi_{i,H}d\Omega
= (A_{H,2,3} \mathbf u_H)_i.
\end{eqnarray}
Finally, the vector $b_{Hh,2,3}$ which is defined as
\begin{eqnarray}\label{b_H_h_2_3}
(b_{Hh,2,3})_i = \int_\Omega \zeta(\widetilde u_h)^3\phi_{i,H}d\Omega,
\end{eqnarray}
will not change neither during the nonlinear iteration process.

Now, let us come to consider the computation for the value $\xi$. It is obvious that
$\xi$ has the following expansion
\begin{eqnarray}
\xi &=&\int_{\Omega}|\nabla\widetilde u_h|^2d\Omega +\int_{\Omega}w (\widetilde u_h)^2d\Omega +
\int_{\Omega}\zeta(u_H+\alpha \widetilde u_h)^2(\widetilde u_h)^2d\Omega\nonumber\\
&=&\int_{\Omega}\big(|\nabla\widetilde u_h|^2 + w (\widetilde u_h)^2\big)d\Omega
+\int_\Omega \zeta\big((u_H)^2+2\alpha u_H\widetilde u_h
+\alpha^2(\widetilde u_h)^2\big)(\widetilde u_h)^2d\Omega\nonumber\\
&:=& d_1 + d_2,
\end{eqnarray}
where
\begin{eqnarray}\label{d_1}
d_1 = \int_{\Omega}\big(|\nabla\widetilde u_h|^2 + w (\widetilde u_h)^2\big)d\Omega,
\end{eqnarray}
and
\begin{eqnarray}\label{d_2}
d_2&=& \sum_{i=1}^{N_H}\sum_{j=1}^{N_H}u_iu_j\int_{\Omega}\zeta(\widetilde u_h)^2\phi_{i,H}\phi_{j,H}d\Omega
+2\alpha \sum_{i=1}^{N_H}u_i\int_\Omega \zeta(\widetilde u_h)^3\phi_{i,H}d\Omega \nonumber\\
&&\ \ \ + \alpha^2\int_\Omega \zeta(\widetilde u_h)^4d\Omega\nonumber\\
&=&\mathbf u_H^TA_{H,2,3}\mathbf u_H+2\alpha \mathbf u_H^Tb_{Hh,2,3}+\alpha^2 \xi_h,
\end{eqnarray}
with the scalar $\xi_h$ being defined as follows
\begin{eqnarray}
\xi_h = \int_\Omega \zeta(\widetilde u_h)^4d\Omega.\label{beta_h}
\end{eqnarray}

Based on above discussion and preparation, we define the following algorithm for solving the nonlinear
eigenvalue problem (\ref{simple_Eigen_Problem}) in Step 2 of Algorithm \ref{Algm:One_Step_Correction}.
\begin{algorithm}\label{Algorithm_Nonlinear_Iteration}
Nonlinear iteration method for eigenvalue problem (\ref{simple_Eigen_Problem})
\begin{enumerate}
\item Preparation for the nonlinear iteration:
Compute the tensor $T_H$ as in (\ref{T_H}),
the matrices $A_{H,1}$ and $A_{H,2,3}$ as in (\ref{A_H_1}) and (\ref{A_H_2_3}),
vectors $b_{Hh,1}$ and $b_{Hh,2,3}$ as in
(\ref{b_H_h_1}) and (\ref{b_H_h_2_3}), scalars $d_1$ and $\xi_h$ as in (\ref{d_1}) and (\ref{beta_h}).

\item Nonlinear iteration:
\begin{enumerate}
\item Produce the matrix $A_{H,2,1}$ and $A_{H,2,2}$ as in (\ref{A_H_2_1})
and (\ref{Tensor_Multiply}). Then compute the
matrix $A_H=A_{H,1}+A_{H,2,1}+2\alpha A_{H,2,2}+\alpha^2 A_{H,2,3}$.
\item Produce $b_{Hh,2,1}$ and $b_{Hh,2,2}$ as in (\ref{Tensor_Multiply_2})
and (\ref{b_H_h_2_2}). Then compute the vector $b_{Hh}=b_{Hh,1}+b_{Hh,2,1}+2\alpha b_{Hh,2,2}+\alpha^2 b_{Hh,2,3}$.
\item Compute the scalar $d_2$ as in (\ref{d_2}). Then compute the scalar
$\xi=d_1+d_2$.
\item Then solve the eigenvalue problem (\ref{Eigenvalue_H_h}) by some normal eigensolver to get a new
eigenfunction $(u_H,\alpha)$ and the corresponding
eigenvalue $\lambda_h$.
\item If the accuracy for nonlinear iteration is satisfied, stop the nonlinear iteration.
Otherwise, continue the  nonlinear iteration.
\end{enumerate}
\item Output the eigenfunction $u_h = u_H + \alpha \widetilde u_h = \sum_{i=1}^{N_H}u_i\phi_{i,H}+\alpha\widetilde u_h$ and the eigenvalue $\lambda_h$.
\end{enumerate}
\end{algorithm}
\begin{remark}\label{Remark_Tensor}
It is obvious that assembling the Tensor, matrices, vectors and scalar in Step 1 of
Algorithm \ref{Algorithm_Nonlinear_Iteration} needs computational
work $\mathcal O(N_h)$. But, the computational work for each nonlinear iteration step (Step 2)
of Algorithm \ref{Algorithm_Nonlinear_Iteration} is only $\mathcal O(M_H)$, where $M_H$ denotes the
computational work for solving the eigenvalue problem (\ref{Eigenvalue_H_h}) and it holds that $M_H\geq N_H$.
Assume there needs $\varpi$ nonlinear iteration times. Then the computational work for
Algorithm \ref{Algorithm_Nonlinear_Iteration} is only $\mathcal O(N_h+\varpi M_H)$.
\end{remark}

\section{Multigrid method for GPE}
Based on the preparation in previous sections, we introduce a type of multigrid method based on the
{\it One Correction Step} defined in Algorithms \ref{Algm:One_Step_Correction} and the implementing technique
defined in Algorithm \ref{Algorithm_Nonlinear_Iteration}.
This type of multigrid method can obtain the same optimal error estimate as
that for solving the GPE directly on the finest finite element space.

In order to develop multigrid scheme, we define a sequence of triangulations
$\mathcal{T}_{h_k}$
of $\Omega$ as follows. Suppose $\mathcal{T}_{h_1}$ is
produced from $\mathcal{T}_H$ by some regular refinements and
let $\mathcal{T}_{h_k}$ be obtained from $\mathcal{T}_{h_{k-1}}$ via a regular
refinement such that 
\begin{equation}\label{refinement index}
h_k\approx\frac{1}{\beta}h_{k-1}, \ \ \ \ k = 2,\ldots,n,
\end{equation}
where $\beta$ denotes the refinement index.
Based on this sequence of meshes, we construct the corresponding linear finite
element spaces $V_{h_1}, \ldots, V_{h_n}$ such that
\begin{eqnarray}\label{fem_set_relationship}
V_{H} = V_{h_0} \subseteq V_{h_1} \subset V_{h_2} \subset \ldots \subset V_{h_n}\subset V.
\end{eqnarray}
In this paper, we assume the following relations of approximation errors hold
\begin{eqnarray}\label{Error_k_k_1}
\eta_a(V_{h_k})\approx \frac{1}{\beta}\eta_a(V_{h_{k-1}}),\ \ \ \
\delta_{h_k}(u)\approx\frac{1}{\beta}\delta_{h_{k-1}}(u),\ \ \ k=2,\ldots,n.
\end{eqnarray}
\begin{algorithm}\label{Algm:Multi_Correction}
Multigrid Scheme for GPE
\begin{enumerate}
\item Construct a sequence of nested finite element spaces $V_{H}, V_{h_1}, V_{h_2},
\ldots, V_{h_n}$ such that (\ref{fem_set_relationship}) and (\ref{Error_k_k_1}) hold.

\item Solve the GPE on the initial finite element space $V_{h_1}$:
Find $(\lambda_{h_1},u_{h_1})\in \mathbb{R}\times V_{h_1}$ such that $b(u_{h_1},u_{h_1})=1$
and
\begin{eqnarray*}\label{Initial_Nonlinear_Eigen_Problem}
a(u_{h_1},v_{h_1})&=&\lambda_{h_1}b(u_{h_1},v_{h_1}),\ \ \ \ \forall v_{h_1}\in V_{h_1}.
\end{eqnarray*}
\item Do $k = 1, \ldots, n-1$\\
Obtain a new eigenpair approximation
$(\lambda_{h_{k+1}},u_{h_{k+1}})\in \mathbb{R}\times V_{h_{k+1}}$
with the one correction step being defined by Algorithm \ref{Algm:One_Step_Correction} and the nonlinear iteration
being defined by Algorithm \ref{Algorithm_Nonlinear_Iteration}
\begin{eqnarray*}
(\lambda_{h_{k+1}},u_{h_{k+1}}) =
{\it Correction}(V_{H},\lambda_{h_k},u_{h_k},V_{h_{k+1}},\varsigma_{h_{k+1}}).
\end{eqnarray*}
End Do
\end{enumerate}
Finally, we obtain an eigenpair approximation $(\lambda_{h_{n}},u_{h_{n}})
\in \mathbb{R}\times V_{h_{n}}$.
\end{algorithm}
The error estimates for Algorithm \ref{Algm:Multi_Correction} can be stated as follows.
\begin{theorem}\label{Thm:Multi_Correction}(\cite[Theorem 4.1,Corollary 4.1]{XieXie_CICP})
Assume $h_1<h_0$ (as in Lemma \ref{lemma:Maday}) and the error $\varsigma_{h_{k+1}}$ of the linear solving by the multigrid method
in the correction step on the $(k+1)$-th level mesh satisfies $\varsigma_{h_{k+1}}\leq \eta_a(V_{h_k})\delta_{h_k}(u)$ for
$k=1,\ldots,n-1$. After implementing Algorithm \ref{Algm:Multi_Correction}, the resultant eigenpair
approximation $(\lambda_{h_n},u_{h_n})$ has following error estimates
\begin{eqnarray}
\|\bar{u}_{h_n}-u_{h_n}\|_1 &\lesssim&\beta^2\eta_a(V_{h_n})\delta_{h_n}(u),\label{Multi_Correction_Err_fun1}\\
\|\bar{u}_{h_n}-u_{h_n}\|_0 &\lesssim& \eta_a(V_{h_n})\delta_{h_n}(u),\label{Multi_Correction_Err_fun0}\\
|\bar{\lambda}_{h_n}-\lambda_{h_n}| &\lesssim& \eta_a(V_{h_n})\delta_{h_n}(u),\label{Multi_Correction_Err_eigen}
\end{eqnarray}
under the condition $C\beta^2\eta_a(V_H)<1$ for the concerned constant $C$.

Furthermore, we have the following optimal error estimates
\begin{eqnarray}
\|u-u_{h_n}\|_1 &\lesssim&\delta_{h_n}(u),\label{Multi_Correction_Err_fun_final}\\
\|u-u_{h_n}\|_0 &\lesssim& \eta_a(V_{h_n})\delta_{h_n}(u),\label{Multi_Correction_Err_fun_final}\\
|\lambda-\lambda_{h_n}| &\lesssim& \eta_a(V_{h_n})\delta_{h_n}(u).\label{Multi_Correction_Err_eigen_final}
\end{eqnarray}
\end{theorem}
Now, we come to estimate the computational work for the multigrid
scheme defined by Algorithm \ref{Algm:Multi_Correction} with the nonlinear iteration method
defined by Algorithm \ref{Algorithm_Nonlinear_Iteration}. Since the linear
boundary value problem (\ref{Aux_Source_Problem}) in Algorithm \ref{Algm:One_Step_Correction} is
solved by multigrid method, the computational work is asymptotically optimal.

First, we define the dimension of each level linear finite element space as
\begin{eqnarray*}
N_k := {\rm dim}V_{h_k},\ \ \ k=1,\ldots,n.
\end{eqnarray*}
Then we have
\begin{eqnarray}\label{relation_dimension}
N_k \thickapprox\Big(\frac{1}{\beta}\Big)^{d(n-k)}N_n,\ \ \ k=1,\ldots,n.
\end{eqnarray}

Different from the method in \cite{XieXie_CICP}, the computational work for the second step
in Algorithm \ref{Algm:One_Step_Correction}  with the nonlinear iteration method in
Algorithm \ref{Algorithm_Nonlinear_Iteration} is $\mathcal O(N_k+\varpi M_H)$
in each level space $V_{h_k}$.
\begin{theorem}\label{Computation_Work_Estimate_Theorem}
Assume solving the linear eigenvalue problem (\ref{Eigenvalue_H_h}) in the
coarse spaces $V_{H,h_k}$ ($k=1,\ldots, n$) and $V_{h_1}$ need work $\mathcal{O}(M_H)$
and $\mathcal{O}(M_{h_1})$, respectively, and the work of the multigrid method for
solving the source problem in $V_{h_k}$ is $\mathcal{O}(N_k)$
for $k=2,3,\ldots,n$. Let $\varpi$ denote the nonlinear iteration times when we solve
the nonlinear eigenvalue problem (\ref{simple_Eigen_Problem}).
Then the work involved in Algorithm \ref{Algm:Multi_Correction} has the following estimate:
\begin{eqnarray}\label{Computation_Work_Estimate}
{\rm Total\ work}&=&\mathcal{O}\big(N_n
+ \varpi M_H\log N_n + \varpi M_{h_1}\big).
\end{eqnarray}
\end{theorem}
\begin{proof}
Let $W_k$ denote the work in the $k$-th finite element space $V_{h_k}$.
Then with the correction definition in Algorithms \ref{Algm:One_Step_Correction}, 
\ref{Algorithm_Nonlinear_Iteration} and Remark \ref{Remark_Tensor}, we have
\begin{eqnarray}\label{work_k}
W_k&=&\mathcal{O}\left(N_k + \varpi M_H\right).
\end{eqnarray}
Iterating (\ref{work_k}) and using the fact (\ref{relation_dimension}), we obtain
\begin{eqnarray}\label{Work_Estimate}
\text{Total work} &=& \sum_{k=1}^nW_k\nonumber =
\mathcal{O}\left(\varpi M_{h_1} +\sum_{k=2}^n
\Big(N_k + \varpi M_H\Big)\right)\nonumber\\
&=& \mathcal{O} \Big(\sum_{k=2}^n N_k+(n-1)\varpi M_H+ \varpi M_{h_1} \Big)\nonumber\\
&=& \mathcal{O}\left(\sum_{k=2}^n
\Big(\frac{1}{\beta}\Big)^{d(n-k)} N_n + \varpi M_H\log N_n+ \varpi M_{h_1} \right)\nonumber\\
&=& \mathcal{O}\big(N_n + \varpi M_H\log N_n + \varpi M_{h_1}\big).
\end{eqnarray}
This is the desired (\ref{Computation_Work_Estimate}) result and we complete the proof.
\end{proof}
\begin{remark}
With the help of the implementing  technique defined in
Algorithm \ref{Algorithm_Nonlinear_Iteration}, the nonlinear iteration times affect
the final computational work by $\varpi M_H$ and $\varpi M_{h_1}$ which is very
small scale since $M_H\ll N_{h_n}$ and $M_{h_1}\ll N_{h_n}$. It means that the final
computational work is asymptotic optimal and depends very weakly on the the nonlinearity of GPE.
\end{remark}
\section{Numerical examples}
In this section, we provided three numerical examples to validate the efficiency of the
multigrid method stated in Algorithm \ref{Algm:Multi_Correction} with the nonlinear iteration
technique defined in Algorithm \ref{Algorithm_Nonlinear_Iteration}.
About the convergence behavior of Algorithm \ref{Algm:Multi_Correction}, please
refer to \cite{JiaXieXieXu,XieXie_CICP} which gives the corresponding
numerical results. Here, we are only concerned with the computing time (in seconds) for
Algorithm \ref{Algm:Multi_Correction} for the eigenvalue problem (\ref{GPEsymply2}) with
different choices of $\zeta$.

\begin{example}\label{Example_1}
In this example, we solve GPE (\ref{GPEsymply2}) with the computing domain
$\Omega$ being the unit square
$\Omega=(0,1)\times (0,1)$, $W=x_1^2+x_2^2$ with different choices of $\zeta$.
\end{example}

The sequence of finite element spaces are constructed by using the linear finite element
on the sequence of meshes which are produced by regular refinement
with $\beta = 2$ (connecting the midpoints of each edge).
In this example, we choose the coarse mesh $\mathcal T_H=\mathcal T_{h_1}$ which is shown in
Figure \ref{Initial_Meshes_Example1} to investigate the CPU time (in seconds) for different $\zeta$.
\begin{figure}[ht]
\centering
\includegraphics[width=5cm,height=5cm]{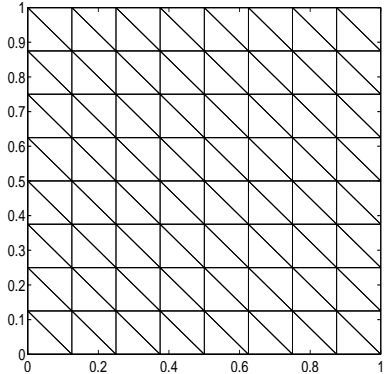}
\caption{\small\texttt The coarse mesh $\mathcal T_H$ 
for Example \ref{Example_1}}
\label{Initial_Meshes_Example1}
\end{figure}

For comparison, we also present the CPU time of the original multigrid method which has been introduced in
\cite{XieXie_CICP}. The CPU time results are shown in Figure \ref{Time_GPE_Model_2D}. From Figure \ref{Time_GPE_Model_2D},
we can find that the computational work of Algorithm \ref{Algm:Multi_Correction} with the nonlinear iteration defined by Algorithm \ref{Algorithm_Nonlinear_Iteration} is much smaller than that of the original multigrid method in \cite{XieXie_CICP}.
The computational work of the the original multigrid method in \cite{XieXie_CICP} has linear scale but depends on the
nonlinearity of the problem. It is well known that bigger value of $\zeta$ means stronger nonlinearity of the problem (\ref{GPEsymply2}).
This is why that the original multigrid method needs more CPU time for bigger $\zeta$. Figure \ref{Time_GPE_Model_2D} also shows that
the asymptotic computational work for Algorithm \ref{Algm:Multi_Correction} is almost independent from the nonlinearity
(the choice of $\zeta$) of the eigenvalue problem (\ref{GPEsymply2}) which consists with the estimate (\ref{Computation_Work_Estimate})
in Theorem \ref{Computation_Work_Estimate_Theorem}.
\begin{figure}[ht]
\centering
\includegraphics[width=6cm,height=6cm]{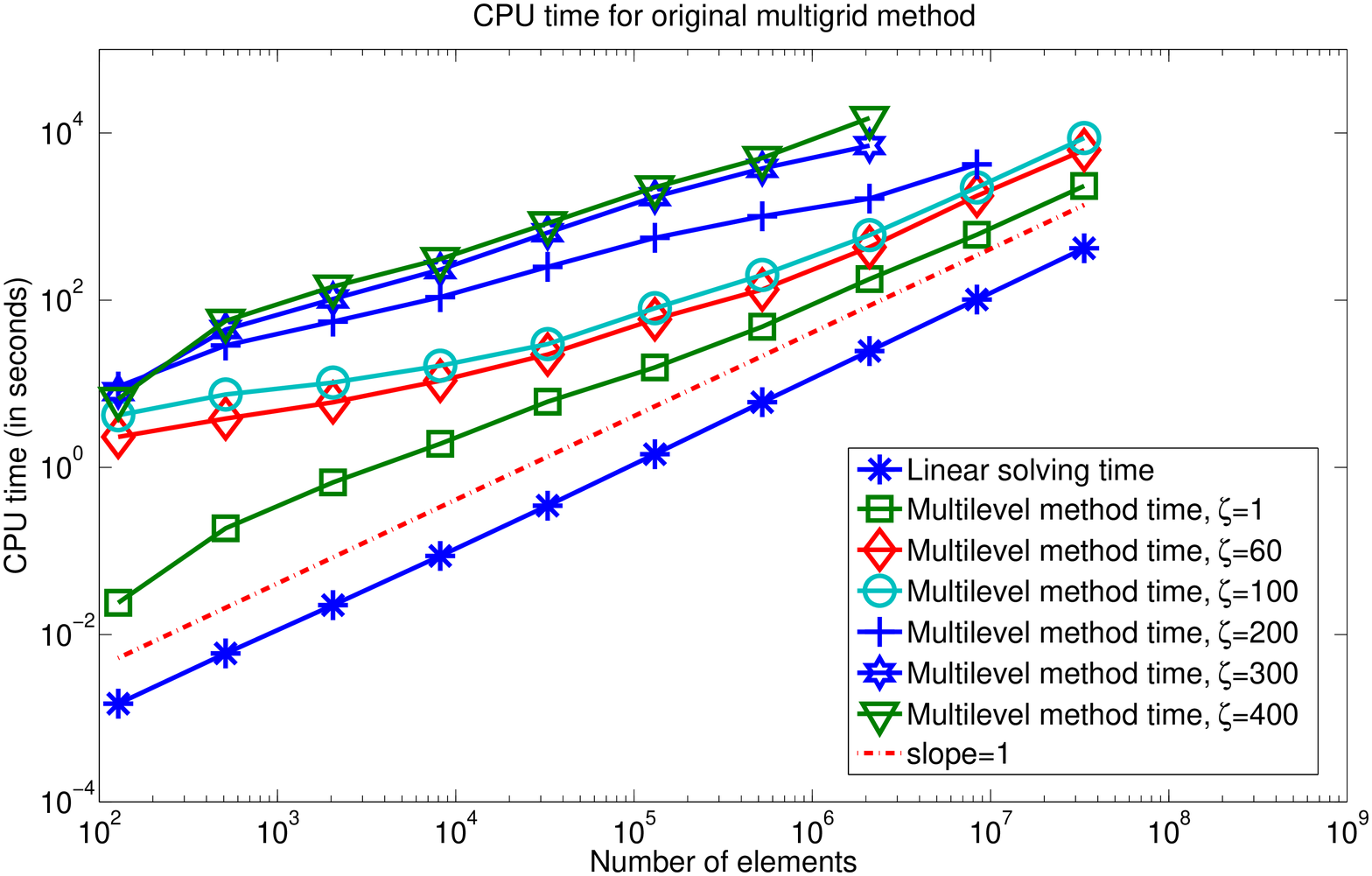}
\includegraphics[width=6cm,height=6cm]{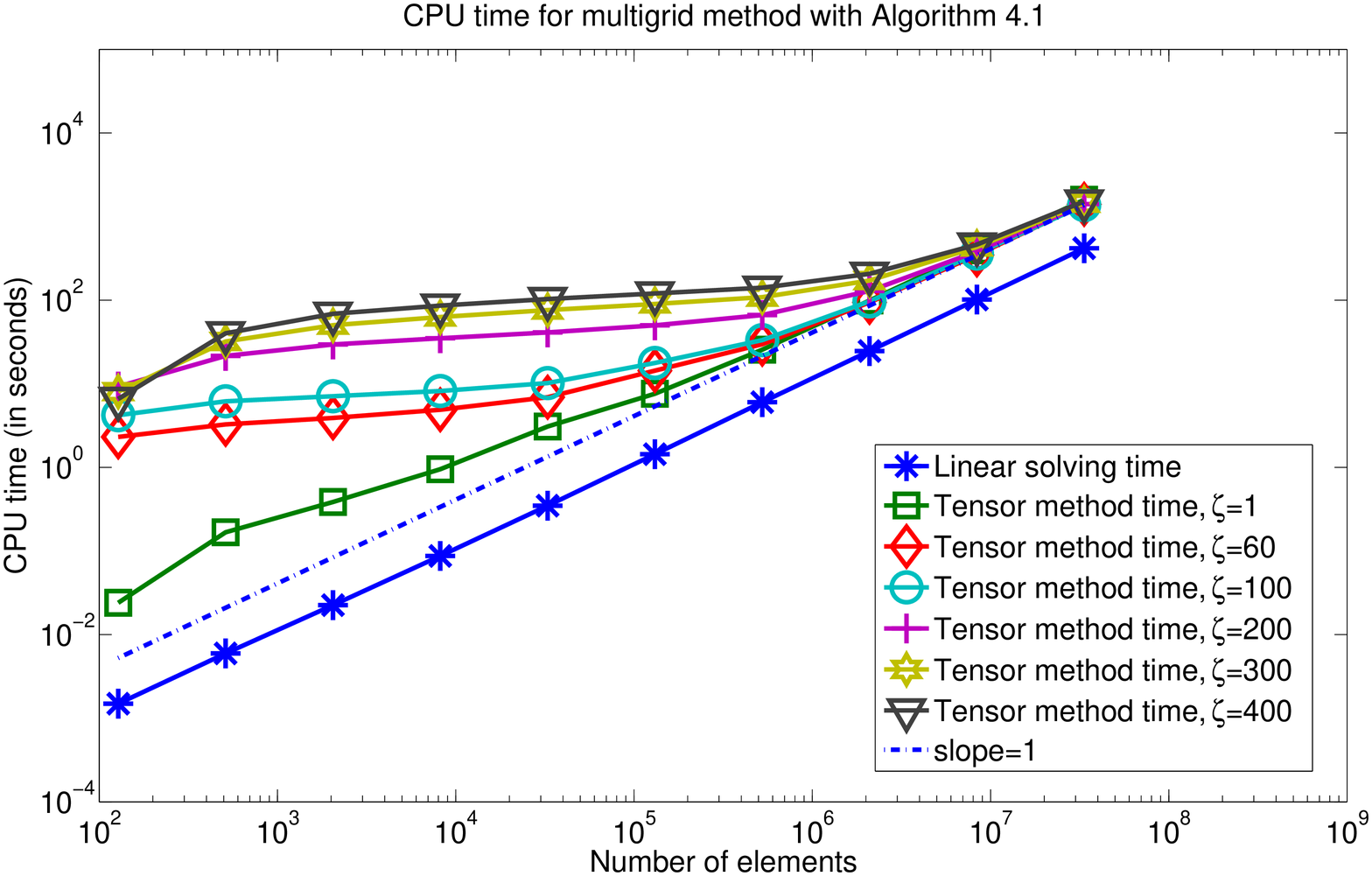}
\caption{\small\texttt The CPU time (in seconds) for two dimensional eigenvalue problem (\ref{GPEsymply2}).
Here linear solving time denotes the CPU time for the linear elliptic boundary value problem by the multigrid method,
multilevel method time denotes the CPU time for the original multigrid method in \cite{XieXie_CICP} and
tensor method time denotes the CPU time for Algorithm \ref{Algm:Multi_Correction}.}\label{Time_GPE_Model_2D}
\end{figure}

\begin{example}\label{Example_2}
In the second example, we solve GPE (\ref{GPEsymply2}) with the computing domain
$\Omega$ being the unit brick $\Omega=(0,1)\times (0,1)\times (0,1)$, $W=x_1^2+x_2^2+x_3^2$ with different choice of $\zeta$.
\end{example}

The sequence of finite element spaces are constructed by using the linear finite element
on the sequence of meshes which are produced by regular refinement
with $\beta = 2$ from the coarse mesh $\mathcal T_H$ which is shown in Figure \ref{Initial_Meshes_Example2}.
In this example, we also use the initial mesh $\mathcal{T}_H = \mathcal{T}_{h_1}$ to investigate the
CPU time (in seconds) for different $\zeta$.

In this example, we also present the CPU time for the original multigrid method introduced in
\cite{XieXie_CICP} for comparison. Figure \ref{Time_GPE_Model_3D} shows the CPU time results where we can find 
the same behavior as in Example \ref{Example_1}. 
The computational work of Algorithm \ref{Algm:Multi_Correction} is much smaller than the original
multigrid method in \cite{XieXie_CICP}. Figure \ref{Time_GPE_Model_3D} shows that
the computational work of the the original multigrid method in \cite{XieXie_CICP}
depends on the strength of the nonlinearity. Furthermore, the asymptotic computational work for
Algorithm \ref{Algm:Multi_Correction} is almost independent of the nonlinearity (the choice of $\zeta$) of
the eigenvalue problem (\ref{GPEsymply2}) which consists with the estimate (\ref{Computation_Work_Estimate})
in Theorem \ref{Computation_Work_Estimate_Theorem}.

\begin{figure}[ht]
\centering
\includegraphics[width=10cm,height=6cm]{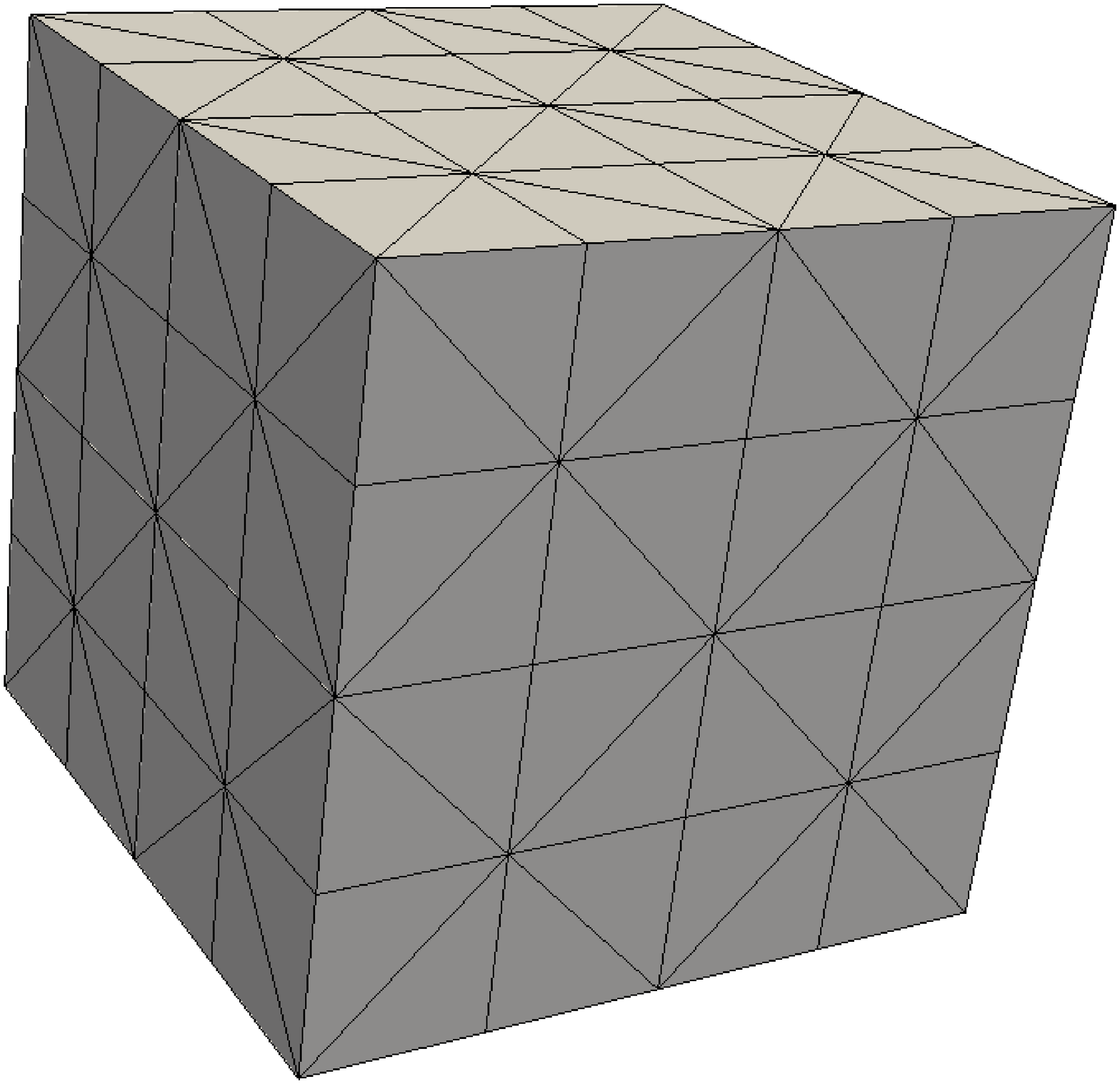}
\caption{\small\texttt The coarse mesh $\mathcal T_H=\mathcal T_{h_1}$ for Example \ref{Example_2}}
\label{Initial_Meshes_Example2}
\end{figure}

\begin{figure}[ht]
\centering
\includegraphics[width=6cm,height=6cm]{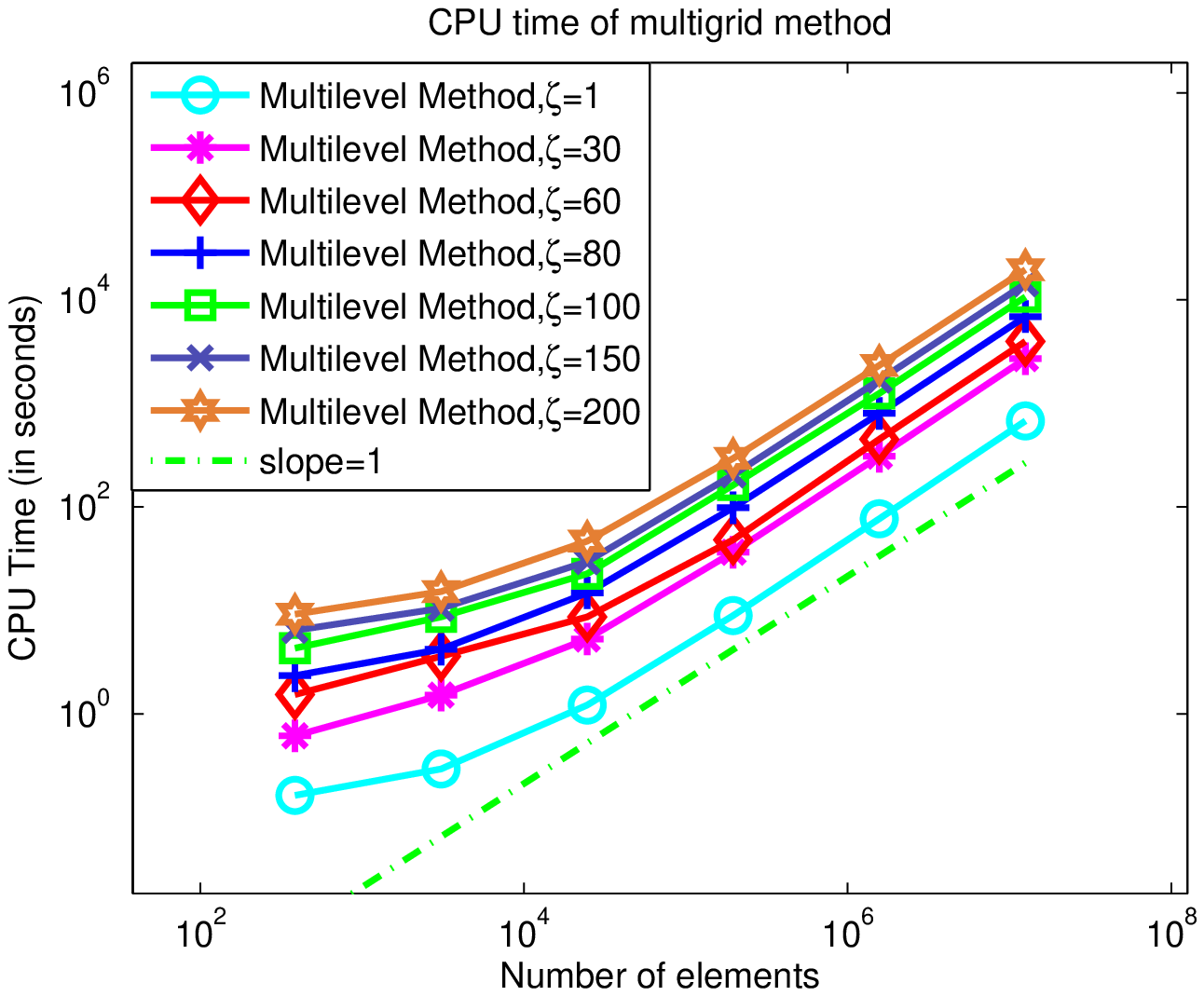}
\includegraphics[width=6cm,height=6cm]{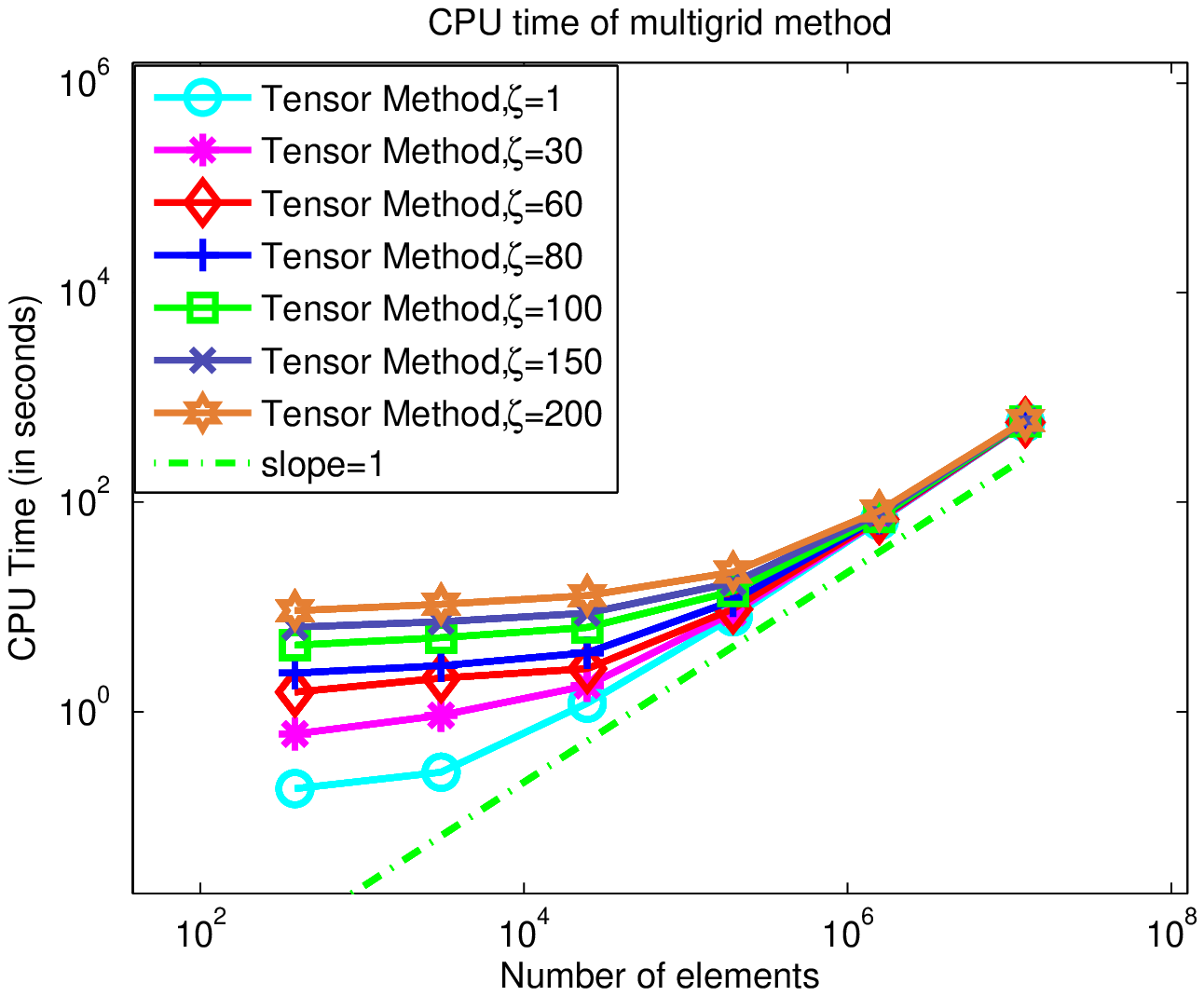}
\caption{\small\texttt The CPU time (in seconds) for three dimensional eigenvalue problem (\ref{GPEsymply2}).
Here multilevel method time denotes the CPU time for the original multigrid method in \cite{XieXie_CICP} and
tensor method time denotes the CPU time for Algorithm \ref{Algm:Multi_Correction}.}\label{Time_GPE_Model_3D}
\end{figure}
\begin{example}\label{Example_3}
In this example, we also solve the GPE (\ref{GPEsymply2}), where the computing domain $\Omega$ is
the $L$-shape domain $\Omega=(0,2)\times(0,2)\backslash[1, 2)\times [1, 2)$,
$W=x_1^2+x_2^2$.
\end{example}
Due to the reentrant corner of $\Omega$, the exact eigenfunction with singularities is
expected. The convergence order for approximate eigenpair is less than the order predicted
by the theory for regular eigenfunctions.  Thus, the adaptive refinement is adopted to couple
with the multigrid method described in Algorithm \ref{Algm:Multi_Correction}.
Since the exact eigenvalue is not known, we also choose an adequately accurate
approximation on a fine enough mesh as the exact one to check the error estimates.
We give the numerical results of the multigrid method in which the sequence of
meshes $\mathcal T_{h_1},\cdots,\mathcal T_{h_n}$ is produced by the adaptive refinement
with the following a posteriori error estimator
\begin{eqnarray}
\eta^2(u_{h_k},K):=h_K^2\| {\mathcal R}_{K}(\lambda_{h_k},u_{h_k})\|_{0,K}^2
+\sum_{e\in\mathcal E_I,e\subset \partial K}h_e\| {\mathcal J}_{e}(u_{h_k})\|_{0,e}^2,
\end{eqnarray}
where  the element residual  ${\mathcal R}_{K}(u_{h_k})$ and the jump residual
${\mathcal J}_e(u_{h_k})$ are defined as follows:
\begin{eqnarray}
&&{\mathcal R}_{K}(\lambda_{h_k},u_{h_k}) := \lambda_{h_k} u_{h_k}
+\Delta u_{h_k}-W u_{h_k}-\zeta |u_{h_k}|^2u_{h_k}, \qquad \text{in } K\in \mathcal T_{h_k}, \\
&&{\mathcal J}_{e}(u_{h_k}) := -\nabla v^+\cdot\nu^+-\nabla v^-\cdot\nu^-
:=[\nabla v]_e\cdot \nu_e, \quad\  \quad\quad\text{on } e\in \mathcal E_I.
\end{eqnarray}
Here  $\mathcal E_I$ denotes the set of interior faces (edges or sides) of $\mathcal T_{h_k}$ and
$e$ is the common side of elements $K^+$ and $K^-$ with the unit outward normals $\nu^+$ and
$\nu^-$, respectively, and $ \nu_e=\nu^- $.

Figure \ref{Error_GPE_Adaptive_Result} shows the corresponding numerical results by Algorithm \ref{Algm:Multi_Correction}
coupled with the adaptive refinement. From the numerical experiment, it is also observed the errors by
Algorithm \ref{Multi_Correction_Err_eigen} is the same as the original multigrid method in \cite{XieXie_CICP} since
the difference between these two algorithms is only the implementing technique.
From Figure \ref{Error_GPE_Adaptive_Result}, we can also find that Algorithm \ref{Algm:Multi_Correction}
can also work on the adaptive family of meshes and obtain the optimal accuracy.

In this example, for comparison, we also present the CPU time for the original multigrid method introduced in
\cite{XieXie_CICP}. The CPU time results are shown in Figure \ref{Time_GPE_Adaptive_Result} which shows
the same behavior as in previous examples. The computational work of Algorithm \ref{Algm:Multi_Correction}
is much smaller than the original multigrid method in \cite{XieXie_CICP}.
Figure \ref{Time_GPE_Adaptive_Result} shows that the computational work
of the the original multigrid method in \cite{XieXie_CICP}
depends on the strength of the nonlinearity. Furthermore, the asymptotic computational work for
Algorithm \ref{Algm:Multi_Correction} is almost independent from the nonlinearity (the choice of $\zeta$) of
the eigenvalue problem (\ref{GPEsymply2}) even on the adaptive family of meshes.

\begin{figure}[ht]
\centering
\includegraphics[width=6cm,height=6cm]{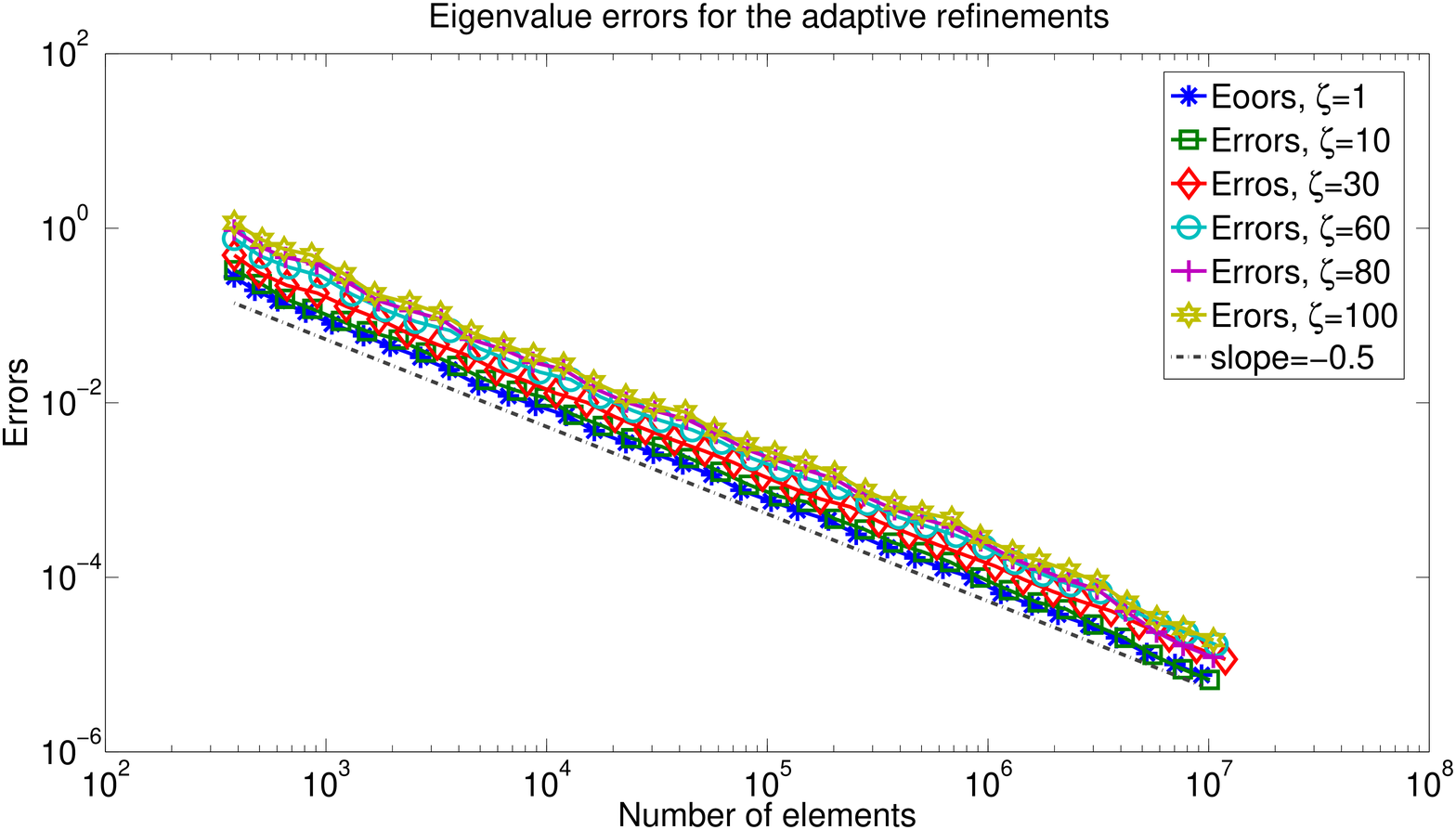}
\includegraphics[width=6cm,height=6cm]{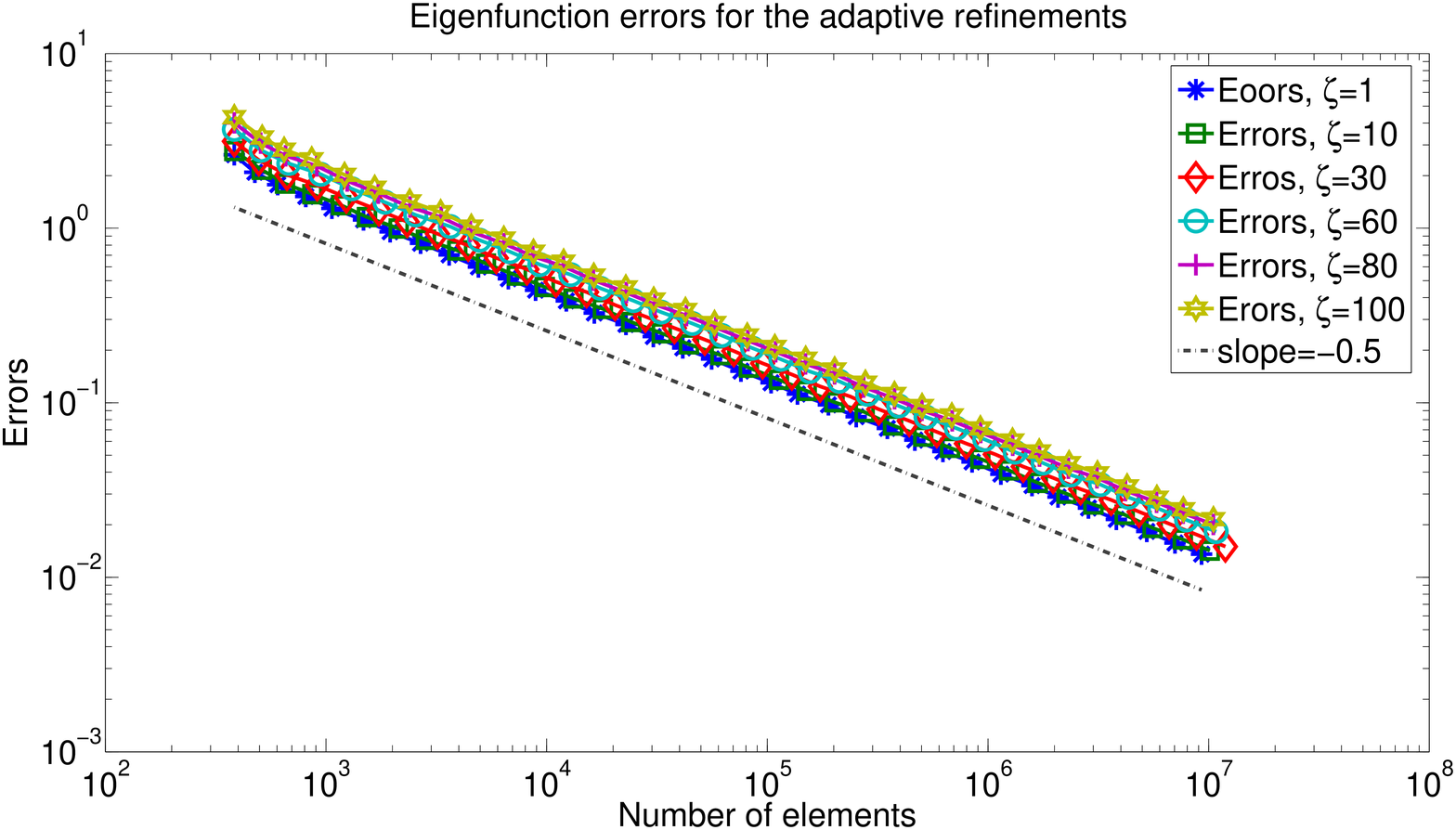}
\caption{\small\texttt The errors for eigenvalue problem (\ref{GPEsymply2})
which is solved by the multigrid method coupled with the adaptive refinement.
The left subfigure shows the errors for the eigenvalue approximation and the right one shows
the posteriori error estimates for the eigenfunction approximations.}\label{Error_GPE_Adaptive_Result}
\end{figure}

\begin{figure}[ht]
\centering
\includegraphics[width=6cm,height=6cm]{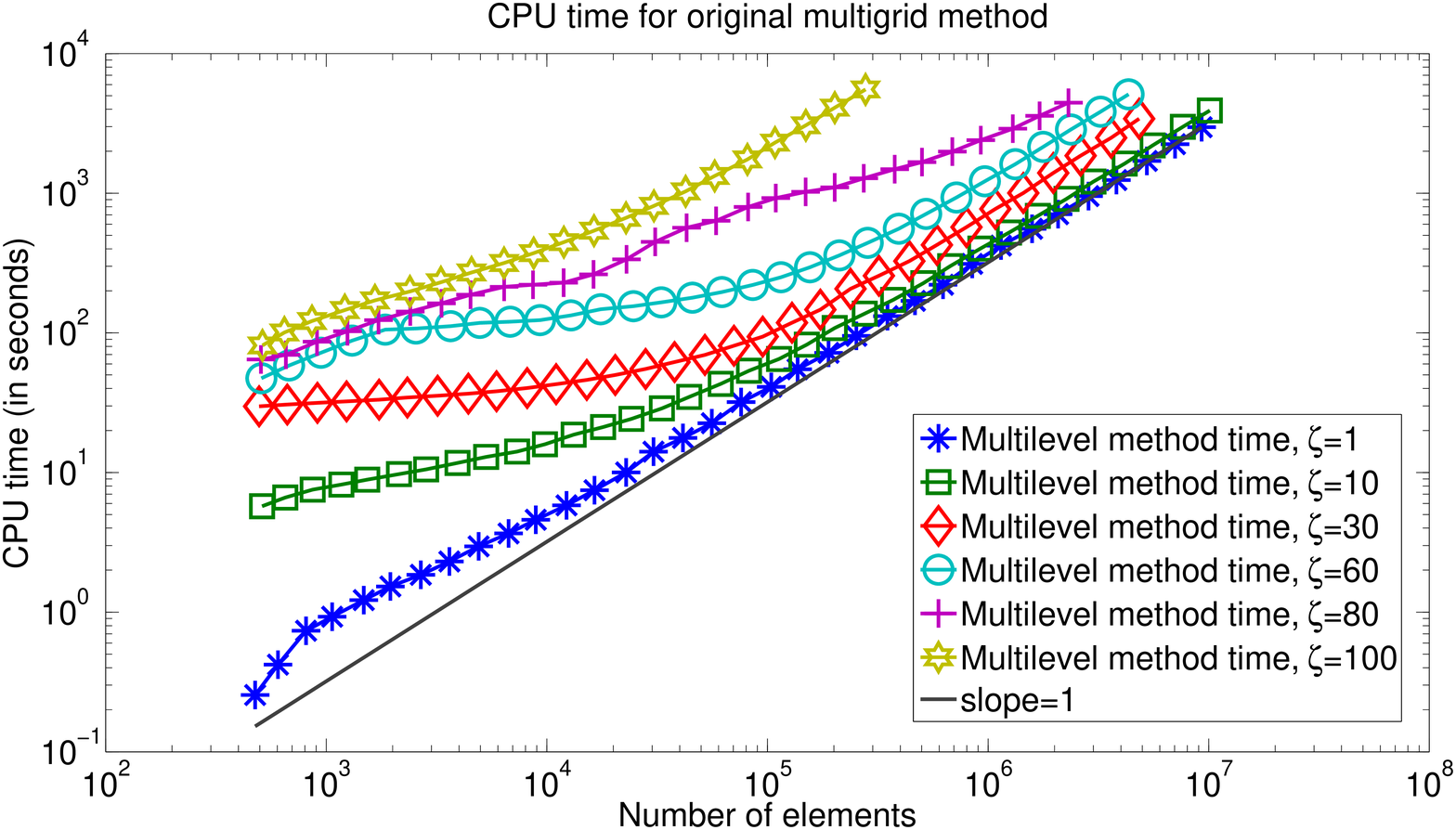}
\includegraphics[width=6cm,height=6cm]{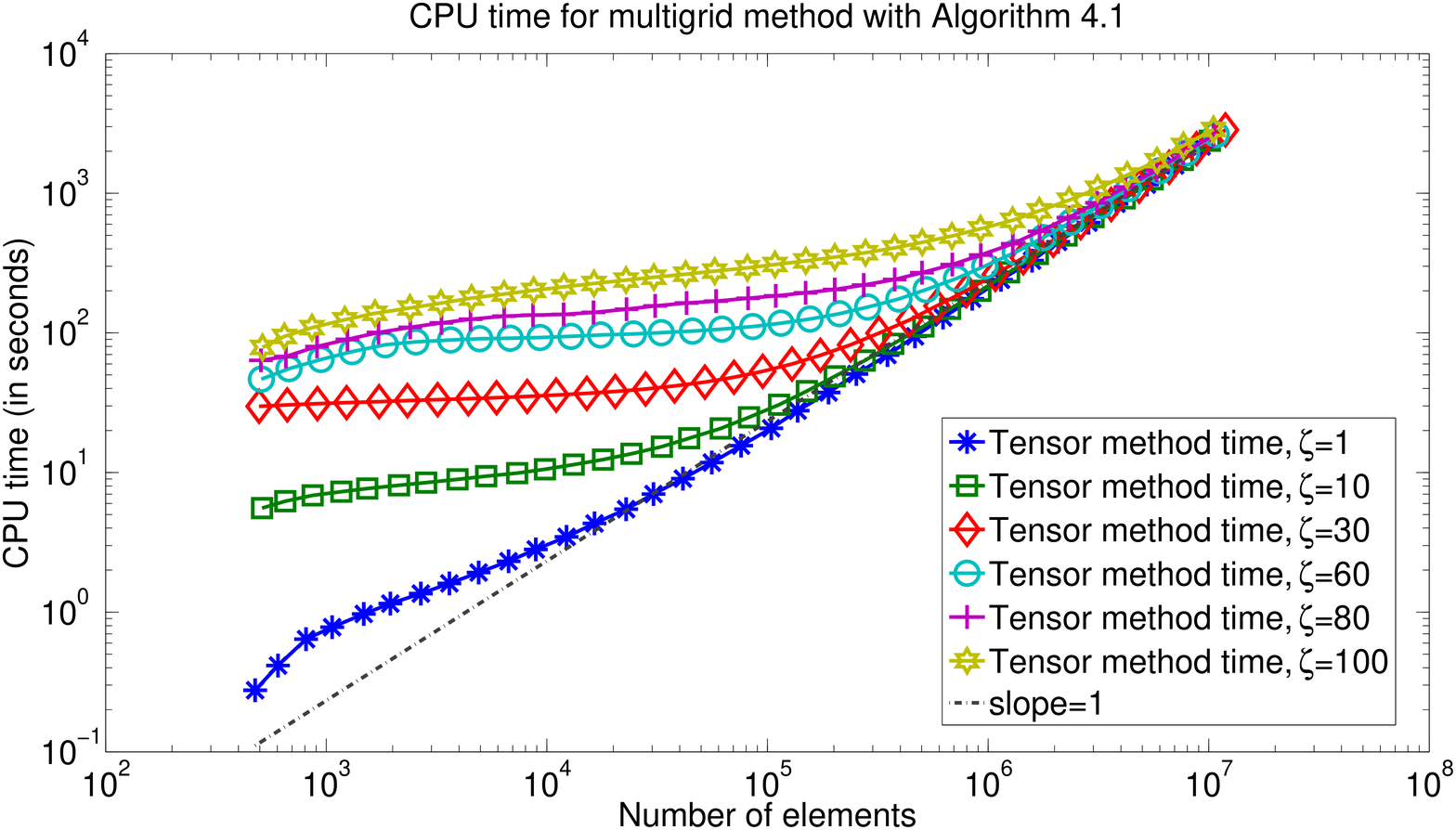}
\caption{\small\texttt The CPU time (in seconds) for  eigenvalue problem (\ref{GPEsymply2})
which is solved by the multigrid method coupled with the adaptive refinement.
Here multilevel method time denotes the CPU time for the original multigrid method in \cite{XieXie_CICP} and
tensor method time denotes the CPU time for Algorithm \ref{Algm:Multi_Correction}.}\label{Time_GPE_Adaptive_Result}
\end{figure}

\section{Concluding remarks}
In this paper, we propose an efficient implementing method for the multigrid method introduced in
\cite{XieXie_CICP} to solve GPE.  With the new implementing method for the nonlinear iteration,
the asymptotical computational work for solving GPE is almost the same as solving the corresponding
linear boundary value problem by the multigrid method, and almost independent of the nonlinearity of GPE.
Three examples are provided to validate the efficiency of the proposed method.

The idea and method here can also be extended to other problems with polynomial nonlinearity such as
Navier-Stokes and some phase models. Furthermore, we can use the algorithms here to design a preconditioner
for the general nonlinear problems and nonlinear eigenvalue problems.




\begin{thebibliography}{17}

\bibitem{Adams}
R. A. Adams, {Sobolev spaces}, Academic Press, New York, 1975.



\bibitem{AnderEnsherMattewWieman}
M. H. Anderson, J. R. Ensher , M. R. Mattews , C. E. Wieman and E. A. Cornell,
{Observation of Bose-Einstein
condensation in a dilute atomic vapor}, Science, 269 (1995), 198-201.

\bibitem{AnglinKetterle}
J. R. Anglin and W. Ketterle, {Bose-Einstein condensation of atomic gasses},
Nature, 416 (2002), 211-218.

%




\bibitem{BaoCai}
W. Bao and Y. Cai, Mathematical theory and numerical methods for Bose-Einstein condestion,
Kinetic and Related Models, 6(1) (2013), 1-135.



%
%
\bibitem{BaoTang}
W. Bao and W. Tang, {Ground-state solution of trapped interacting
Bose-Einstein condensate by directly minimizing the energy functional},
J. Comput. Phys., 187 (2003), 230-254.


\bibitem{Bramble}
J. H. Bramble, {Multigrid Methods}, Pitman Research Notes in Mathematics,
V. 294, John Wiley and Sons, 1993.



%


\bibitem{BrennerScott}
S. Brenner and L. Scott, {The Mathematical Theory of Finite Element
Methods}, New York: Springer-Verlag, 1994.

\bibitem{CancesChakirMaday}
E. Canc\`{e}s, R. Chakir, Y. Maday, {Numerical analysis of
nonlinear eigenvalue problems},
J. Sci. Comput., 45(1-3) (2010), 90-117.


\bibitem{ChienHuangJengLi}
C.-S. Chien, H.-T. Huang,  B.-W. Jeng and Z.-C. Li,
{Two-grid discretization schemes for nonlinear Schr\"{o}inger
equations}, J. Comput. Appl. Math., 214 (2008), 549-571.


\bibitem{ChienJeng}
C.-S. Chien, B.-W. Jeng,
{A two-grid discretization scheme for semilinear elliptic eigenvalue problems},
 SIAM J. Sci. Comput., 27(4) (2006), 1287-1304.




\bibitem{Ciarlet}
P. G. Ciarlet, {The Finite Element Method for Elliptic Problems},
Amsterdam: North-Holland, 1978.



\bibitem{CornellWieman}
E. A. Cornell and C. E. Wieman, {Nobel Lecture: Bose-Einstein condensation
in a dilute gas, the first 70 years and some recent experiments},
Rev. Mod. Phys., 74 (2002), 875-893.



\bibitem{DalGioPitaString}
F. Dalfovo, S. Giorgini, L. P. Pitaevskii and S. Stringari,
{Theory of Bose-Einstein condensation in trapped
gases}, Rev. Mod. Phys., 71 (1999), 463-512.

%



\bibitem{Gross}
E. P. Gross, {Nuovo}, Cimento., 20 (1961), 454.

\bibitem{Hackbush}
W. Hackbush, {Multi-grid Methods and Applications},
Springer-Verlag, Berlin, 1985.


\bibitem{HenningMalqvistPeterseim}
P. Henning, A. M{\aa}lqvist and D. Peterseim, {Two-level discretization techniques for
ground state computations of Bose-Eistein condensates}, SIAM J. Numer. Anal., 52(4) (2014), 1525-1550.


\bibitem{JiaXieXieXu}
S. Jia, H. Xie, M. Xie and F. Xu,
A full multigrid method for nonlinear eigenvalue problems, Sci China Math, 59 (2016), 2037-2048.


\bibitem{Ketterle}
W. Ketterle, {Nobel lecture: When atoms behave as waves:
Bose-Einstein condensation and the atom laser},
Rev. Mod. Phys., 74 (2002), 1131-1151.

\bibitem{LaudauLifschitz}
L. Laudau and E. Lifschitz, {Quantum Mechanics: non-relativistic theory},
Pergamon Press, New York, 1977.

\bibitem{LiebSeiYang}
E. H. Lieb, R. Seiringer and J. Yangvason, {Bosons in a trap:
a rigorous derivation of the Gross-Pitaevskii energy functional},
Phys. Rev. A, 61 (2000), 043602.

\bibitem{LinXie}
Q. Lin and H. Xie, A multi-level correction scheme for eigenvalue problems,
Math. Comp.,  84 (2015), 71-88.




\bibitem{McCormick}
S. F. McCormick, ed., {Multigrid Methods}.
SIAM Frontiers in Applied Matmematics 3.
Society for Industrial and Applied Mathematics, Philadelphia, 1987.




\bibitem{Xie_Steklov}
H. Xie, A type of multilevel method for the Steklov eigenvalue problem,
IMA J. Numer. Anal.,34(2) (2014), 592-608.

\bibitem{Xie_Nonconforming}
H. Xie, A type of multi-level correction scheme for eigenvalue problems
by nonconforming finite element methods, BIT Numer. Math., 55 (2015), 1243-1266.

\bibitem{Xie_JCP}
H. Xie, A multigrid method for eigenvalue problem, J. Comput. Phys., 274 (2014),
550-561.

\bibitem{XieXie_CICP}
H. Xie and M. Xie, A multigrid method for ground state solution of Bose-Einetein condensates,
Commun. Comput. Phys., 19(3) (2016), 648-662.

\bibitem{Xu}
J. Xu, Iterative methods by space decomposition and subspace
correction, SIAM Review, 34(4) (1992), 581-613.

\bibitem{ZhouBEC}
A. Zhou, An analysis of fnite-dimensional approximations
for the ground state solution of Bose-Einstein condensates,
Nonlinearity, 17 (2004), 541-550.

\end{thebibliography}
\end{document}